\documentclass{article} 
\usepackage[a4paper,bindingoffset=0.2in, left=1.25in,right=1.25in,top=1in,bottom=1in, footskip=.25in]{geometry}
\usepackage{amsmath,amssymb,amsfonts,amsthm}
\usepackage{graphicx}
\usepackage{graphics}
\usepackage{multirow}
\usepackage[colorlinks,citecolor=blue]{hyperref}
\usepackage{algorithm}
\usepackage[misc]{ifsym} 
\usepackage{float}
\usepackage{subfig}
\usepackage{algpseudocode}
\usepackage{indentfirst}
\usepackage{enumerate}
\usepackage{authblk}

\newtheorem{theorem}{Theorem}[section]

\newtheorem{proposition}[theorem]{Proposition}

\newtheorem{definition}[theorem]{Definition}
\newtheorem{remark}[theorem]{Remark}

\newcommand{\Rmn}{\mathbb{R}^{m \times n}}
\DeclareMathOperator*{\argmin}{arg\,min}

\newcommand{\N}{\mathbb{N}}
\DeclareMathOperator*{\minimize}{minimize}
\newcommand{\R}{\mathbb{R}}
\newcommand{\Xs}{X^{\star}}
\algnewcommand{\IfThen}[2]{\State \algorithmicif\ #1\  \algorithmicthen\ #2}

\title{A two-phase rank-based algorithm for low-rank matrix completion}
\author[1]{Tacildo de S. Araújo\thanks{e-mail: tacildo.araujo@ifam.edu.br}}
\author[2]{Douglas S. Gonçalves\thanks{e-mail: douglas@mtm.ufsc.br}}
\author[3]{Cristiano Torezzan\thanks{e-mail: torezzan@unicamp.br}}
\affil[1]{IMECC, University of Campinas, Campinas, Brazil}
\affil[2]{CFM, Federal University of Santa Catarina, Florianópolis, Brazil}
\affil[3]{FCA, University of Campinas, Limeira, Brasil}
\date{}    

\begin{document}
\maketitle

\begin{abstract}

Matrix completion aims to recover an unknown low-rank matrix from a small subset of its entries. 
In many applications, the rank of the unknown target matrix is known in advance. 
In this paper, first we revisit a recently proposed rank-based heuristic for ``known-rank'' matrix completion 
and establish a condition under which the generated sequence is quasi-Fejér convergent to the solution set. 
Then, by including an acceleration mechanism similar to Nesterov's acceleration, we obtain a new heuristic. 
Even though the convergence of such heuristic cannot be granted in general, it turns out that it can be very useful as a warm-start phase,  
providing a suitable estimate for the regularization parameter and a good starting-point, to an accelerated Soft-Impute algorithm. 
Numerical experiments with both synthetic and real data show that the resulting two-phase rank-based algorithm can recover low-rank matrices, 
with relatively high precision, faster than other well-established matrix completion algorithms.

\end{abstract}
\ \\


\section{Introduction}\label{sec:intro}


The problem of recovering missing entries in a low-rank matrix $A \in \Rmn$ can be formulated in terms of a rank minimization problem as
\begin{equation}\label{rmp}
	\begin{aligned}
		\minimize_{X \in \Rmn} & \quad \text{rank}(X)  \\
		\text{subject to} & \quad P_{\Omega}(X)=P_{\Omega}(A),
	\end{aligned}
\end{equation}
where $ \Omega $ denotes the set of indices of the known entries of $ A $ and $ P_{\Omega}(\cdot) $ is the projection operator, defined as
$$
[P_{\Omega}(X)]_{ij}:=
\begin{cases}
	X_{ij}, & \textrm{if} \ (i,j) \in \Omega \\ 
	0, & \text{otherwise},
\end{cases}
$$
with $ P_{\Omega}^{\perp}(\cdot) $ defined by $ P_{\Omega}^{\perp}(X)=X- P_{\Omega}(X)$.

Despite its theoretical importance, problem (\ref{rmp}) is non-convex and combinatorially hard for general sets $ \Omega $ \cite{srebro2003weighted}. To overcome such disadvantage, several alternatives have been proposed in the literature \cite{candes2010power, fornasier2011low, tanner2016low}. A common way to swerve the non-convexity in problem  (\ref{rmp}) is to replace the rank objective by a convex relaxation such as the nuclear norm $\|X\|_{\ast}$, as proposed in  \cite{candes2009exact} and \cite{fazel2002matrix}.

The nuclear norm of a matrix is derived from its Singular Value Decomposition (SVD). Let $X = U \Sigma V^\top$ be the SVD of $X \in \mathbb{R}^{m \times n}$ 
and assume that $\sigma_1 \geq \sigma_2 \geq \dots \geq \sigma_{\min\{m,n\}} \geq 0$ are its singular values. The nuclear norm of $X$ is defined as $ \lVert X \rVert_{\ast} := \sum_j \sigma_j $ and it has been used to propose a convex relaxation for problem (\ref{rmp}). 

The methods studied in this paper rely on a deflated version of the SVD decomposition, calculated by using the so-called \textit{Soft-Thresholding} (ST) operator, defined as
\begin{equation}\label{eq:soft}
	S_{\tau} (M) := U \Sigma_{\tau} V^\top, \quad \Sigma_{\tau} = \mbox{diag}[ (\sigma_1 - \tau)_{+},\cdots, (\sigma_q - \tau)_{+}], 
\end{equation}
where $M = U \Sigma V^\top $ is the compact-SVD of a rank $q$ matrix $M$ and $t_{+}=\max(0,t)$. 
It turns out that $S_{\tau} (M)$ is a proximal operator \cite[Theorem~2.1]{cai2010singular} which solves the problem
\[
\min_{X} \dfrac{1}{2}\lVert M-X \rVert_{F}^{2} + \tau \lVert X \rVert_{\ast}. 
\]
In \cite{cai2010singular}, based on the Uzawa's method for finding saddle points of the Lagrangian, the authors present an algorithm, called \textit{Singular Value Thresholding} (SVT), and proved that the sequence generated by 
\begin{align}
X^{k+1} & = S_{\tau} (Y^k)  \\
Y^{k+1} & = Y^k + t_k P_{\Omega} (A - X^{k+1})
\end{align}
where $Y^0 = 0$ and $t_k$ is a step-size, converges to the unique solution of the following optimization problem 
\begin{equation}\label{svt}
	\begin{aligned}
		\minimize_{X \in \Rmn} & \quad  \tau \lVert X \rVert_{\ast}  + \dfrac{1}{2} \lVert X \rVert_{F}^2  \\
		\text{subject to} & \quad  P_{\Omega}(X)=P_{\Omega}(A),
	\end{aligned}
\end{equation}
where $ \tau > 0 $ is a regularization parameter. The component $\tau \lVert X \rVert_{\ast}$ in the objective function is a convex relaxation for $\text{rank}(X)$, while $\frac{1}{2}\lVert X \rVert_{F}^2$ is a strongly convex term granting \eqref{svt} a unique solution.  
Due to its theoretical and computational properties, the SVT algorithm is an important reference for matrix completion and it is often used as a benchmark.

Another alternative formulation for problem (\ref{rmp}) is to consider a tolerance on the recovering of the known entries. This can be particularly useful in applications where data are obtained through noisy processes.
In this case, it may be worth to consider the following optimization problem  
\begin{equation}\label{nnmp}
	\begin{aligned}
		\minimize_{X \in \Rmn} & \quad \lVert X \rVert_{\ast}   \\
		\text{subject to} & \quad  \lVert P_{\Omega}(X-A)\rVert_{F}^{2}\leq \delta,
	\end{aligned}
\end{equation}
where $\delta \geq 0 $ is a given recovering error tolerance.

In  \cite{ma2011fixed} and also in \cite{mazumder2010spectral} it is exploited the following Lagrangian formulation for problem (\ref{nnmp}),
\begin{equation}\label{nnm}
	\minimize_{X\in \Rmn} \quad  \dfrac{1}{2}\lVert P_{\Omega}(A)- P_{\Omega}(X)\rVert_{F}^{2} + \lambda \lVert X \rVert_{\ast} =: f_{\lambda}(X),
\end{equation}
where $ \lambda > 0 $ is a regularization parameter. The authors in \cite{ma2011fixed} showed that the sequence produced by
\begin{align}
X^{k+1} & = S_{\lambda t_k} (Y^k)  \\
Y^{k+1} & = X^{k+1} + t_k P_{\Omega} (A - X^{k+1})
\end{align}
converges to a solution of (\ref{nnm}). This iterative process is called \textit{Fixed Point Continuation} (FPC). 
If the step-size is fixed as $t_k=1$, the above iteration reduces to the so-called Soft-Impute (SI) algorithm discussed in \cite{mazumder2010spectral}. 
Both, SI and FPC, rely on a pre-specified decreasing sequence of regularization parameters $\lambda_1 > \dots > \lambda_K$, 
solving \eqref{nnm}, up to a predetermined tolerance, for each value of $\lambda$.  



Although they may have different motivations, the FPC and SI algorithms can be seen as particular cases of the proximal gradient method \cite{fista}. This fact was actually used in \cite{yao2018} to derive a convergence analysis for these algorithms and to propose acceleration strategies for the SI. 

Such algorithms, however, are still very sensitive to the choice of the regularization parameter $\lambda$ 
and the parameter tuning process can be quite cumbersome in real applications. 
Another information that is disregarded, or not properly used, by these algorithms is the eventual knowledge of the rank of the target matrix. In some applications, such as in problems involving Euclidean Distance Matrices (EDM), the rank of the matrix to be completed is known in advance. For example, it can be proved that the rank of an EDM derived from a set of points in $\R^{d}$ is at most $d+2$ \cite{dokmanic2015euclidean}. This information might be useful to estimate the parameter $\lambda$ and improve the completion performance.

In \cite{moreira2018novel} the authors take into account the rank information and propose an algorithm called \textit{Fixed-Rank Soft-Impute} (FRSI) to complete missing entries in EDMs using the rank information to estimate the regularization parameter $\lambda$. However, despite the good numerical results, no convergence analysis was provided for FRSI. 


In this paper, we first revisit the rank-based heuristic proposed in \cite{moreira2018novel} and analyze some properties of the operator defining the iterative process. We show that under an assumption on the behavior of the singular values of the iterates, the sequence generated by such heuristic is quasi-Fejér convergent to the set of matrices with rank not greater than the target rank and that agree with the target matrix in the sampled entries. 
Then, based on acceleration techniques for proximal gradient methods \cite{parikh2014proximal}, we devise an accelerated heuristic.

Even though the convergence of such heuristic cannot be granted in general, it turns out that it can be very useful as a warm-start phase,  
to find a suitable estimate for the regularization parameter $\lambda$ and a good starting-point, to an accelerated Soft-Impute algorithm \cite{yao2018}. 
The resulting is a two-phase rank-based algorithm for low-rank matrix completion that presents promising results in numerical experiments with both synthetic and real data.



The rest of the paper is organized as follows. 
Section~\ref{sec:frsi} revisits the algorithm proposed in \cite{moreira2018novel} and proves its convergence, in the quasi-Fejér sense, 
under an assumption on the behavior of the singular values of the iterates.  
Since the required assumption is strong, convergence is not granted in general. 
However, in Section~\ref{sec:mcri}, we discuss how such heuristic can be used as a warm-start phase in a two-phase algorihtm: 
the heuristic provides a starting-point and the value for the regularization parameter to be used in an Accelerated Soft-Impute algorithm in the second phase.  
Section~\ref{sec:numexp} reports some numerical experiments on synthetic and real data and compares the proposed two-phase algorithm with other well-established matrix completion algorithms. Section~\ref{sec:fim} brings some concluding remarks and discuss directions for future investigations.


\section{Revisiting fixed-rank Soft-Impute (FRSI) and its convergence}\label{sec:frsi}

Let $A \in \Rmn$ be a matrix with missing entries and rank $r$, which we assume it is known in advance. 

Let us review FRSI \cite{moreira2018novel} by first recalling the iteration of Soft-Impute (SI) \cite{mazumder2010spectral}. 
According to the notation used in the Introduction, SI can be described as
\begin{equation}\label{eq:si}
	X^k = S_{\lambda}\left(P_{\Omega}(A)+ P_{\Omega}^{\perp}(X^{k-1})\right).
\end{equation}
This iteration can be deduced by applying the proximal gradient method to the compositive convex optimization problem \eqref{nnm}, as discussed in Appendix~\ref{sec:sipg}.

Notice that for an arbitrary value of $\lambda$, there is no reason to expect $X^k$ to have rank $r$. 
However, if we set 
$$
\lambda = \sigma_{r+1} \left(P_{\Omega}(A)+ P_{\Omega}^{\perp}(X^{k-1})\right),
$$
i.e., the $r+1$ largest singular value of $P_{\Omega}(A)+ P_{\Omega}^{\perp}(X^{k-1})$, then from the definition of $S_{\lambda}$ in \eqref{eq:soft} 
it follows that the rank is at most $r$ for each iterate $X^k$. 
This observation motivated the FRSI algorithm proposed in \cite{moreira2018novel}. 
However, no convergence analysis was provided in that paper.

By defining $g(X) = \frac{1}{2}\| P_{\Omega}(X - A) \|_F^2$, we can write FRSI iteration as:
\begin{equation}\label{eq:frsi}
	X^k = S_{\sigma_{r+1} \left( X^{k-1} - \nabla g(X^{k-1}) \right)} \left( X^{k-1} - \nabla g(X^{k-1} ) \right).
\end{equation}
Here, we provide some insights on the convergence of iteration \eqref{eq:frsi}, by analyzing the operator

\[
T(X) := S_{\sigma_{r+1}(X - \nabla g(X))} \left( X - \nabla g(X) \right).
\]

\begin{proposition}
Let $T: \Rmn \rightarrow \Rmn$ be the operator defined above and assume that $\text{rank}(A) = r$. 
Then, the following properties hold. 
\begin{enumerate}[(i)]
\item $T(A) = A$
\item If $B \in \Rmn$ is such that $\text{rank}(B) \leq r$ and $P_{\Omega}(B) = P_{\Omega}(A)$, then $T(B) = B$
\item $T(P_{\Omega}^\perp (A)) = A$
\item  $T(0) = T(P_{\Omega} (A))$
\end{enumerate}
\end{proposition}

\begin{proof}	
Observe that $\nabla g(X) = P_{\Omega}(X - A)$. (i) Thus, since $\nabla g(A) = 0$, and $\text{rank}(A)=r$, it is straightforward that $T(A) = A$, 
i.e, $A$ is a fixed-point of $T$. The same reasoning applies to a matrix $B$ such that $P_{\Omega}(B) = P_{\Omega}(A)$ and $\text{rank}(B) \leq r$,
proofing (ii).  

Also notice that 
\[
P_{\Omega}^\perp (A) - \nabla g(P_{\Omega}^\perp (A)) = P_{\Omega}^\perp (A) - P_{\Omega} \left( P_{\Omega}^\perp (A) - A  \right) = P_{\Omega}^\perp (A) + P_{\Omega} (A) = A,
\]
and thus $T(P_{\Omega}^\perp (A)) = A$ as well, showing (iii). Finally, since
\[
P_{\Omega} (A) - \nabla g(P_{\Omega} (A)) = P_{\Omega} (A) - P_{\Omega} \left( P_{\Omega} (A) - A \right) = P_{\Omega} (A) = 0 - P_{\Omega} \left( 0 - A \right) = 0 - \nabla g(0),
\]
we conclude (iv): $T(0) = T(P_{\Omega} (A))$.
\end{proof}

Therefore, not only the target matrix $A$ is a fixed-point of $T$ but any other matrix $B$, of rank at most $r$, such that $P_{\Omega}(B) = P_{\Omega}(A)$. 
Perhaps, more surprisingly, is the fact that $T$ also admits fixed-points $X$, such that $P_{\Omega}(X) \ne P_{\Omega}(A)$, as show the next proposition.

\begin{proposition}\label{prop:badX}
Let $X \in \Rmn$ be a matrix of rank at most $r$, with truncated ($r+1$)-SVD $X = U \Sigma V^\top$. 
If $\nabla g(X) = - \gamma U V^\top - U_{\perp} \Sigma_{\perp} V_{\perp}^\top$, 
where the columns of $U_{\perp}$ and $V_{\perp}$ are orthonormal bases for the orthogonal complement of range of $U$ and $V$, respectively, 
and $\gamma >0$ with $\sigma_i^{\perp} < \gamma$, for $i=r+2,\dots, \min \{m,n\}$, then $X = T(X)$. 
\end{proposition}
\begin{proof}
Observe that
\[
X - \nabla g(X) = U (\Sigma + \gamma \mbox{I}_{r+1}) V^\top + U_{\perp} \Sigma_{\perp} V_{\perp}^\top,
\]
then, since $\sigma_{r+1}(X - \nabla g(X)) = \gamma > \sigma_i^\perp$, we obtain $T(X) = U \Sigma V^\top =X$. 
\end{proof}

From the above propositions, we see that although the target matrix $A$ is a fixed point of $T$, which is desirable, in general the operator will not have a unique fixed point and more, there may be fixed points $X$ such that $P_{\Omega}(X) \ne P_{\Omega}(A)$. 
Thus, we cannot expect $T$ to be a contraction.

Nevertheless, we shall see that if the sequence of singular values $\sigma_{r+1}(X^k - \nabla g(X^k))$ goes to zero fast enough, then we can prove that the sequence $\{X^k\}$ is quasi-Fejér convergent to the set
\[
{\cal X}^* = \{ X \in \Rmn \mid \text{rank}(X) \leq r, P_{\Omega}(X) = P_{\Omega}(A) \}. 
\]

\begin{definition}\label{def:qfejer}
A sequence $\{ X^k \}$ in $\Rmn$ is quasi-Fejér convergent to $C \subset \Rmn$ if, for each $\Xs \in C$, there exists a non-negative summable sequence $\{ \varepsilon_k \}$ such that
\[
\| X^k - \Xs \| \leq \| X^{k-1} - \Xs \| + \varepsilon_k, \quad k=1,2,\dots 
\]
\end{definition}

\begin{proposition}\label{prop:qfejer}
Let $C \subset \Rmn$ be a nonempty set and $\{ X^k \}$ a quasi-Fejér sequence convergent to $C$. Then, 
\begin{enumerate}[(i)]
\item $\{X^k \}$ is bounded.
\item If $\{X^k \}$ has a cluster point $\bar{X} \in C$, then the whole sequence $\{X^k \}$ converges to $\bar{X}$.
\end{enumerate}
\end{proposition}
\begin{proof}
See \cite[Proposition~1]{Iusem2003}. 
\end{proof}

\begin{theorem}\label{teo:qfejer}
Let $\{ X^k \}$ be the sequence generated by $X^k = T(X^{k-1})$, with $X^0 \in \Rmn$. 
If the sequence $\{ \sigma_{r+1}(X^k - \nabla g(X^k)) \}$ is summable, then $\{ X^k \}$ is quasi-Fejér convergent to the set ${\cal X}^*$.
\end{theorem}
\begin{proof}
Let $\Xs \in {\cal X}^*$. Consider the notation $\|\cdot\| = \|\cdot\|_F$,  $\sigma_{r+1}^k = \sigma_{r+1}(X^k - \nabla g(X^k))$ and $\sigma_i^{\star} = \sigma_i(\Xs)$. 
Then, 
\begin{align}
\| X^{k+1} - \Xs \| & = \| T(X^k) - T(\Xs) \| \nonumber \\ 
\ & = \| S_{\sigma_{r+1}(X^k - \nabla g(X^k))}(X^k - \nabla g(X^k)) - S_{\sigma_{r+1}(\Xs)}(\Xs) \| \nonumber \\
\ & \leq \| S_{\sigma_{r+1}(X^k - \nabla g(X^k))}(X^k - \nabla g(X^k)) - S_{\sigma_{r+1}(X^k - \nabla g(X^k))}(\Xs) \|  \nonumber \\
\ & \quad + \| S_{\sigma_{r+1}(X^k - \nabla g(X^k))}(\Xs) - S_{\sigma_{r+1}(\Xs)}(\Xs) \| \nonumber \\
\ & \leq \| X^k - \nabla g(X^k) - \Xs \| + \left( \sum_{i=1}^r ( (\sigma_i^{\star} - \sigma_{r+1}^k)_{+} - \sigma_i^{\star} )^2 \right)^{1/2} \nonumber \\
\ & \leq \| P_{\Omega}^\perp (X^k - \Xs)  \| + \sqrt{r} \sigma_{r+1}^k  \leq \| X^k - \Xs  \| + \sqrt{r} \sigma_{r+1}^k \label{eq:fejer} 
\end{align}
where we have used the triangle inequality and the nonexpansive property of $S_{\lambda}(\cdot)$, with $\lambda = \sigma_{r+1}(X^k - \nabla g(X^k))$ fixed. 
Hence, if the sequence $\{ \sigma_{r+1}^k \}$ is summable, then $\{ X^k \}$ is quasi-Fejér convergent to ${\cal X}^*$.
\end{proof}

Therefore, as long as
\begin{equation}\label{eq:condsum}
\sum_{k=0}^{\infty} \sigma_{r+1}(X^k - \nabla g(X^k))  < \infty,
\end{equation}
$\{X^k \}$ generated by \eqref{eq:frsi} will be quasi-Fejér convergent to ${\cal X}^*$ and, according to Proposition~\ref{prop:qfejer}(ii), if it has a cluster point in this set, the whole sequence will converge to it. Unfortunately, condition \eqref{eq:condsum} is admittedly strong, and does not hold in general. 
For this reason, FRSI, defined by iteration \eqref{eq:frsi}, should be regarded as an heuristic.

\section{A two-phase rank-based algorithm}\label{sec:mcri}

Although the iterative process \eqref{eq:frsi} may not converge to a matrix in ${\cal X}^*$, 
here we propose to use it, for a fixed number of iterations, as a ``warm-start'' phase to obtain a good starting-point 
and an estimate to the regularization parameter $\lambda$ (see problem \eqref{nnm}) before applying the Soft-Impute algorithm (see \eqref{eq:si}). 

This is motivated by our numerical experience with the iterative process \eqref{eq:frsi}: 
we observed that when $\{ X^k \}$ does not converge to an element in ${\cal X}^*$, 
it usually converges to an $\tilde{X}$ as in Proposition~\ref{prop:badX} which, although $P_{\Omega}(\tilde{X}) \ne P_{\Omega}(A)$, 
is such that $\| \tilde{X} \|_* < \| A \|_*$, suggesting $\tilde{X}$ as a minimizer of $\frac{1}{2}\| P_{\Omega} (X) -  P_{\Omega}(A) \|_F^2 + \lambda \|X\|_*$ 
for an appropriate value of $\lambda>0$.

First, inspired by accelerated versions of the proximal gradient method \cite{fista}, 
we include an acceleration for FRSI heuristic, resulting in Algorithm~\ref{alg:phase1}. 
This warm-start phase will be called {\it Phase One}. 


\begin{algorithm}
	\caption{Phase One: Warm-Start}\label{alg:phase1}
	Input: Known entries of $ A \in \Rmn$ indexed by $ \Omega $, rank $r$, $\epsilon > 0$,  $w \in \N$, and $ \beta > 0$.  
	\begin{algorithmic}[1]
		\State Initialize $X^{0}=0$, $ Z^{1}=0$ and $\rho_0 = \infty$ 
		\For {$ j=1 $ to $ w $ }
		\State Compute the truncated $(r+1)$-SVD of $ P_{\Omega}(A)+ P_{\Omega}^{\perp}(Z^{j})$
		\State Set $\rho_j = \sigma_{r+1}(P_{\Omega}(A)+ P_{\Omega}^{\perp}(Z^{j})) $ \label{step:prox1}
		\IfThen{$\lvert\rho_j - \rho_{j-1}\rvert/(1 + \rho_{j-1})< \epsilon_{\rho}$}{exit.} \label{step:stable}
		\State Compute $ X^{j} \leftarrow S_{\rho_j}\left(P_{\Omega}(A)+ P_{\Omega}^{\perp}(Z^{j})\right)$
		\State $ Z^{j+1} \leftarrow X^{j} + \dfrac{j-1}{j+\beta} \left(X^{j}-X^{j-1}\right)$ \label{step:acc1}
		\EndFor
	\end{algorithmic} 
	Output: $Z^j, \rho_j$
\end{algorithm} 

Phase one runs for a pre-specified number $w$ of iterations or until the values of $\rho_j = \sigma_{r+1}(P_{\Omega}(A)+ P_{\Omega}^{\perp}(Z^{j}))$ stabilize. The last value of $\rho_j$ from phase one is used as regularization parameter $\lambda$ for the second phase, which consists of an accelerated Soft-Impute algorithm for problem \eqref{nnm}, starting from $Z^{j+1}$. {\it Phase Two} is described in Algorithm~\ref{alg:phase2}.


\begin{algorithm}
	\caption{Phase Two: Accelerated Soft-Impute}\label{alg:phase2}
	Input: Known entries of $ A \in \Rmn$ indexed by $ \Omega $, rank $r$, $\epsilon > 0$,  $it_{\max} \in \N$, $\lambda > 0$ and $X^0 \in \Rmn$
	\begin{algorithmic}[1]
		\State Initialize $Z^{1}=X^0$
		\For {$ k= 1$  to $ it_{\max} $ }
		\State Compute $ X^{k} \leftarrow S_{\lambda}\left(P_{\Omega}(A)+ P_{\Omega}^{\perp}(Z^{k})\right)$ \label{step:prox2}
		\IfThen{some stopping criterion is verified}{stop.}
		\State $ Z^{k+1} \leftarrow X^{k} + \dfrac{k-1}{k+2} \left(X^{k}-X^{k-1}\right)$ \label{step:acc2}
		\EndFor		 
	\end{algorithmic} 
	Output: $X^{k}$
\end{algorithm} 

\begin{remark}
Differently from Phase one, where a truncated ($r+1$)-SVD was sufficient to evaluate the thresholding operator (because the threshold value was exactly the $r+1$ largest singular value of $P_{\Omega}(A)+ P_{\Omega}^{\perp}(Z^{k})$), in Phase Two the value of $\lambda$ is fixed and may be different from  $\sigma_{r+1} (P_{\Omega}(A)+ P_{\Omega}^{\perp}(Z^{k}))$. As a result, we need to keep an estimate of the rank $r_k$, which is updated in each iteration (starting with $r_1=r$). We compute a truncated ($r_k+1$)-SVD of $P_{\Omega}(A)+ P_{\Omega}^{\perp}(Z^{k})$. If the $r_k+1$ singular value is already below the threshold $\lambda$, we keep the rank estimate $r_k$. Otherwise, we increase $r_k$ (to $r_k + 5$, for example) and repeat the truncated SVD. Finally, $r_{k+1}$ is set to the number of positive shifted singular values after the last truncated SVD. A similar scheme was used in \cite{cai2010singular}. 
\end{remark}

\begin{algorithm}
	\caption{Two-phase rank-based algorithm}\label{alg:main}
	Input: Known entries of $ A \in \Rmn$ indexed by $ \Omega $, rank $r$, $\epsilon > 0$,  $w, it_{\max} \in \N$, and $ \beta > 0$.  
	\begin{algorithmic}[1]
		\State Call Algorithm~\ref{alg:phase1} passing $A$, $\Omega$, $r$, $\epsilon>0$, $w$ and $\beta>0$ \hfill  $\triangleright$ {\it Warm-start}
		\State Set $\lambda = \rho_j$, $X^0 = Z^j$
		\State Call Algorithm~\ref{alg:phase2} passing $A$, $\Omega$, $r$, $\epsilon>0$, $it_{\max}$, $\lambda$ and $X^0$ \hfill  $\triangleright$ {\it Accelerated Soft-Impute}
	\end{algorithmic} 
	Output: $X^{k}$
\end{algorithm} 



Algorithm~\ref{alg:main} summarizes the two-phase rank-based algorithm which uses Algorithm~\ref{alg:phase1} as a warm-start phase (Phase One) and then calls an Accelerated Soft-Impute (Algorithm~\ref{alg:phase2}) in the second phase. As we will see in the numerical experiments of Section~\ref{sec:numexp}, Algorithm~\ref{alg:main} not only outperforms a previous Fixed-Rank Soft-Impute algorithm \cite{moreira2018novel}, but is also competitive with well-established algorithms for low-rank matrix completion.

Furthermore, Algorithm~\ref{alg:main} has granted convergence to a solution of 
\begin{equation}\label{eq:theproblem}
\min_{X \in \Rmn} \quad \frac{1}{2}\| P_{\Omega} (X) -  P_{\Omega}(A) \|_F^2 + \rho_j \|X\|_*
\end{equation}
(where $\rho_j$ is output of Phase One), because Phase One runs for a finite number of iterations and Phase Two is an accelerated proximal gradient method applied to \eqref{eq:theproblem} (see Appendix~\ref{sec:sipg}).




\section{Numerical results}\label{sec:numexp}
In this section, we perform matrix completion experiments with both, synthetic data and the \textit{MovieLens}\footnote[2]{A data set which has been widely used in matrix completion experiments and is available in https://grouplens.org/datasets/movielens/.} data set. 
Moreover, we also provide an empirical study for choosing the acceleration parameter $\beta$ in Phase One (Algorithm~\ref{alg:phase1}). 

All the algorithms were implemented in Matlab language and all the numerical results were performed on a PC with Intel Core i7-7500U CPU and 16 GB RAM.

The proposed Algorithm~\ref{alg:main} is compared with those mentioned in Section~\ref{sec:intro}: FRSI, SVT and FPC. 
All these methods use PROPACK package \cite{larsen1998lanczos} (more specifically, the routine {\tt lansvd} which implements a variant of Lanczos algorithm designed for large matrices with sparse plus low-rank structure) for computing only the leading singular values/vectors. 

Concerning the stopping criteria for Algorithm~\ref{alg:phase2}, we set 
$$ 
\min \left\lbrace \frac{\lvert f_{\lambda}(X^{k})-f_{\lambda}(X^{k+1})\rvert}{f_{\lambda}(X^{k})}, \frac{\lVert X^{k+1}-X^{k} \rVert_{F}}{\lVert X^{k} \rVert_{F}} \right\rbrace \leq \epsilon_{\lambda},
$$
for a given tolerance $\epsilon_\lambda > 0$ and $f_{\lambda}$ is from \eqref{nnm}. 
For FRSI algorithm we use 
$$
\min \left\lbrace \frac{\lVert P_{\Omega}\left(X^{k}-A\right)\rVert_{F}}{\lVert P_{\Omega} \left(A\right) \rVert_{F}}, \frac{\lVert X^{k+1}-X^{k} \rVert_{F}}{\lVert X^{k} \rVert_{F}} \right\rbrace \leq \epsilon_{1} $$ as the stopping criterion and for SVT and FPC algorithms we follow the recommendations in \cite{cai2010singular} and \cite{ma2011fixed}, and use $ \frac{\lVert P_{\Omega}\left(X^{k}-A\right)\rVert_{F}}{\lVert P_{\Omega} \left(A\right) \rVert_{F}} \leq \epsilon_{2} $, and $ \frac{\lVert X^{k+1}-X^{k}\rVert_{F}}{\max \left\lbrace 1,\lVert X^{k} \rVert_{F}\right\rbrace } \leq \epsilon_{3} $ as the stopping criterion, respectively. 

The following procedure were used for generating the synthetic data set: we generated $ n \times n $ matrices of rank $ r \ll n $ of the form $ A = MN \in \mathbb{R}^{n \times n} $, where the entries of $ M \in \mathbb{R}^{n \times r} $ and $ N \in \mathbb{R}^{r \times n} $ are sampled i.i.d from the standard normal distribution. Then, we deleted, uniformly at random, a percentage $ p $ of entries (unobserved entries) of $ A $.

Before we present some numerical results for both synthetic data and MovieLens, 
we shall give an overview of how to set the parameter $ \beta $ in Algorithm~\ref{alg:phase1}. 

\subsection{Tuning the parameter $\beta$}\label{sec:beta}
To assess the sensitivity of Algorithm~\ref{alg:phase1} to the parameter $\beta$, we performed extensive numerical experiments on the synthetic data set. 
We set a budget of $w = 1,000$ iterations and vary the problem dimension $n$, the rank $r$, percentage of missing data $p$ and the tolerance $\epsilon_\rho$. The experiments show that the number of iterations of Algorithm~\ref{alg:phase1} (Phase One) can be highly reduced by a suitable choice of the parameter $ \beta $, mainly when the number of observed entries is very small. 

Figure \ref{figure1and2} (a) shows the optimal value for $ \beta $ considering the percentage of missing data $ p \in \left\lbrace 92\%, 85\%, 72\%, 50\% \right\rbrace $, $ n = 1000 $, $ r = 5 $, and $ \epsilon_{\rho} = 10^{-8} $. As can be seen, for $ \beta \geq 19 $, Algorithm~\ref{alg:phase1} reaches the minimum number of  iterations in the four scenarios. Furthermore, for $ p = 92\% $ the number of iterations is reduced by 79\% with respect to $\beta=2$ (the default value). On the other hand, for  $ n \in \left\lbrace 500, 1000, 2000, 4000 \right\rbrace $, $ r = 5$, $ \epsilon_{\rho} = 10^{-8} $, and $ p = 40\%, $ Figure~\ref{figure1and2} (b) shows the minimum number of iterations in all scenarios for the same value of $ \beta $ ($ \beta \geq 20 $). 

\begin{figure}
	\subfloat[\label{Aa}]{ \includegraphics[width=0.48\textwidth]{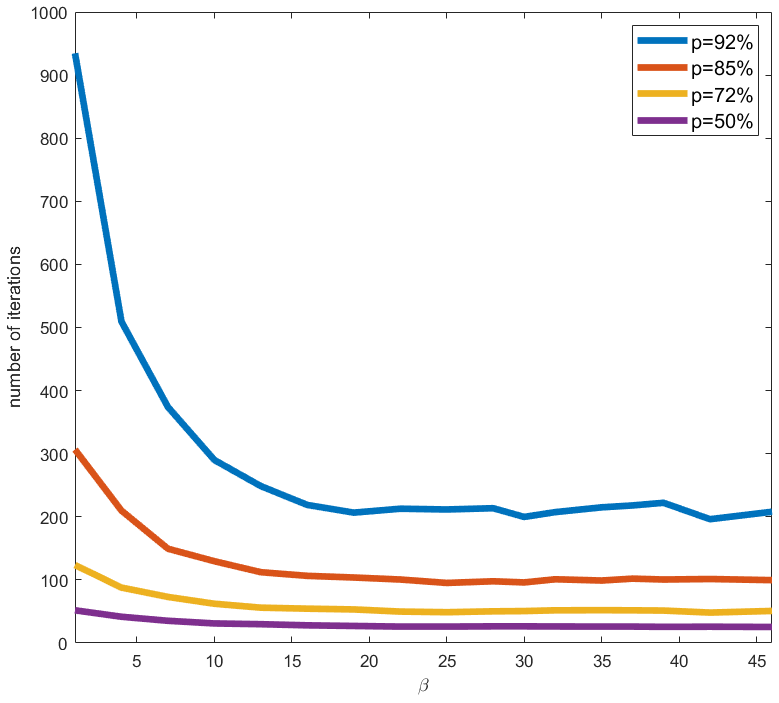}
	}\hfill
	\subfloat[\label{Ab}]{ \includegraphics[width=0.48\textwidth]{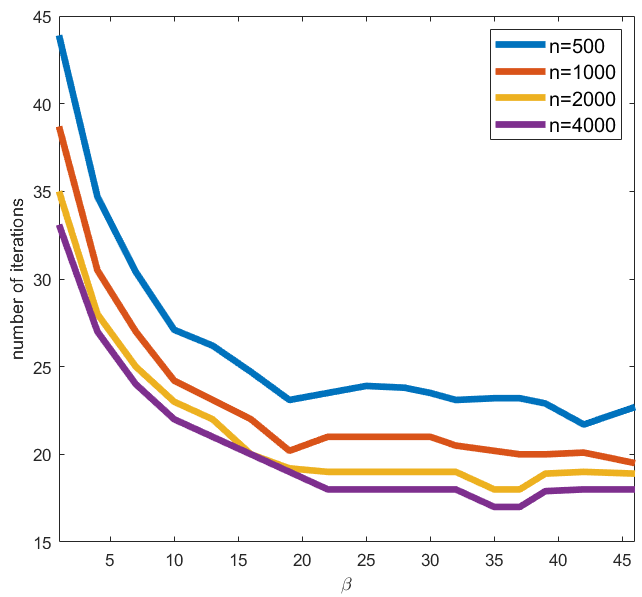}
	}
	\caption{Number of iterations vs $ \beta $: (a) $ n = 1000 $, $ r=5 $, $ \epsilon_{\rho} = 10^{-8} $ and $ p \in \left\lbrace 92\%, 85\%, 72\%, 50\% \right\rbrace$; (b) $ n \in \left\lbrace 500, 1000, 2000, 4000 \right\rbrace $, $ r = 5$, $ \epsilon_{\rho} = 10^{-8} $, and $ p = 40\%$.} 		
	\label{figure1and2}
\end{figure}

Figure~\ref{figura3} gives us an overview of how to set the value of $ \beta $ as the rank $r$ varies. For this experiment, we fixed $ \left(n, \epsilon_{\rho}, p \right) = \left(1000, 10^{-5}, 50\%  \right) $ and $ r \in \left\lbrace 3, 5, 10, 30, 50, 80, 100 \right\rbrace $. As can be seen, the bigger is rank of the target matrix the smaller is the ``optimal'' value of $ \beta $. 

\begin{figure}
	\centering
	\includegraphics[width=.7\textwidth]{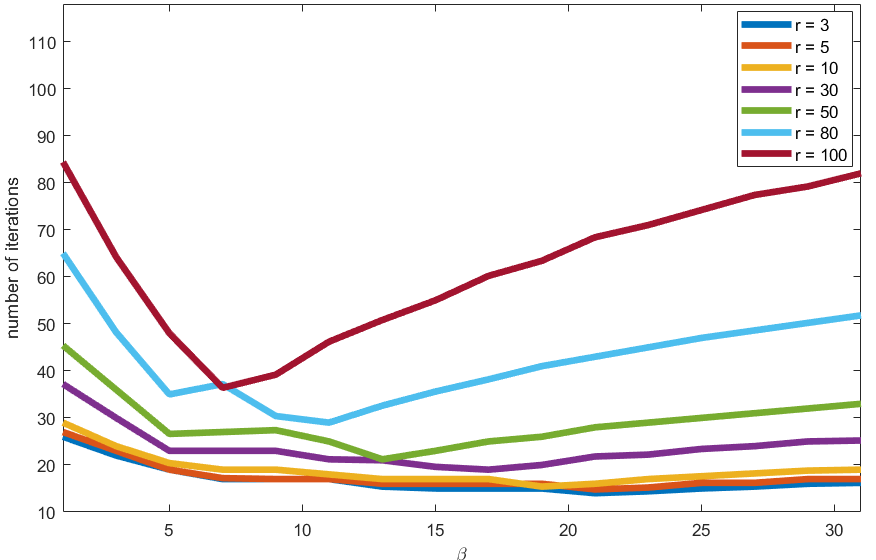}
	\caption{ {\small Optimal value for $ \beta $ with $ n = 1000 $, $ r \in \left\lbrace 3, 5, 10, 30, 50, 80, 100 \right\rbrace $, $ \epsilon_{\rho} = 10^{-5} $ and $ p = 50\% $.}}
	\label{figura3}
\end{figure}


\subsection{Experiments with synthetic data}\label{sec:random}

Now we turn our attention to experiments with synthetic data, generated as described in the beginning of Section~\ref{sec:numexp}. 
For these experiments, we set $ n= 1000 $, $ p = 40\% $ and the rank $r$ takes values in the set $ \left\lbrace 10, 15, 20, 40, 80, 100 \right \rbrace $. 
In the stopping criteria, we used the tolerances $ \epsilon_{\rho} =\epsilon_{1}=\epsilon_{2}= 10^{-4} $, $ \epsilon_{3} = 10^{-3} $, $ \epsilon_{\lambda}=10^{-6} $  and for SVT, following \cite{cai2010singular}, we fixed $ \tau = 5n $ and $ t_k = 1.2n^{2}/ \lvert \Omega \rvert $, where $\lvert\Omega\rvert$ is the cardinality of $ \Omega $. For FPC we have used the standard strategy to update the regularization parameter: $ \lambda_{0} = \lVert P_{\Omega}(A) \rVert_{2} $, $ \lambda_{k}=\max\left \lbrace 0.25\lambda_{k-1}, 0.01 \right\rbrace  $, and $t_k = 1.99$, as recommended in \cite{ma2011fixed}. 
In Algorithm~\ref{alg:main}, we set the maximum number of iterations of phase one as $w=500$ as well as the iteration budget for phase two $it_{\max} = 500$. 
For the acceleration parameter $\beta$ we have used $\{13,13,12,10,5,5\}$, respectively (following the study of Section~\ref{sec:beta}).

For performance evaluation, we use the relative error, defined by  $ \text{Rer}= \lVert A - \tilde{A} \rVert_{F}/\lVert A \rVert_{F}$, 
where $ \tilde{A} $ is the recovered matrix and $A$ is the target one. The experimental results are averaged over $ 5 $ repetitions.

Results are shown in Table \ref{tabela1}, where $r$ denotes the target rank, IT is the total number of iterations 
(for Algorithm~\ref{alg:main}, it is the sum of iterations of the two phases) and $t(s)$ the time in seconds. In this first set of experiments we point out that all algorithms recovered correctly the underlying rank.
As can be seen, our algorithm converges faster than the other algorithms. Furthermore, the bigger is the rank $ r $ of the desired matrix the better is the performance of Algorithm~\ref{alg:main}, when compared with FRSI which, in its turn, is consistently faster than SVT and FPC. 
In terms of relative error Algorithm~\ref{alg:main} was always the first or the second best.

We point out that most of the iterations of Algorithm~\ref{alg:main} correspond to phase one (warm-start) iterations. 
After it switches to the second phase, only a few more iterations are required to reach the stopping criteria. 
On average, for this set of experiments, the number of phase two iterations is less than 10.

\begin{table}
	\caption{Comparison of Algortihm~\ref{alg:main}, FRSI, SVT, and FPC. Performance evaluation for $ n= 1,000 $, $ p = 40\% $, $ r $ takes values in the set  $ \left\lbrace 10, 15, 20, 40, 80, 100 \right \rbrace $ and $\beta \in \{13,13,12,10,5,5\}$, respectively.}\label{tabela1}
	\centering
	\scriptsize 
	\begin{tabular}{cc}
		\begin{tabular}{c|c|ccc}
			\hline
			$ r $  & method & IT    & t(s)          & Rer  \\ \hline
			\multirow{4}{*}{10}  & Alg.~\ref{alg:main} & \textbf{16} & \textbf{1.80}  & \textbf{5.84e-06} \\ & FRSI  & 18  & 2.04   & 1.68e-04  \\ & SVT & 43  & 5.78   & 1.09e-04  \\ & FPC & 74  & 9.68    & 1.70e-05  \\ \hline
			\multirow{4}{*}{15}  & Alg.~\ref{alg:main} & \textbf{18} & \textbf{1.77}  & \textbf{6.90e-06}  \\ & FRSI  & 20  & 2.15   & 1.49e-04  \\ & SVT & 47  & 5.82   & 1.07e-04  \\ & FPC & 81  & 13.04   & 1.72e-05  \\ \hline
			\multirow{4}{*}{20}  & Alg.~\ref{alg:main} & \textbf{18} & \textbf{1.80}   & \textbf{1.12e-06}  \\ & FRSI  & 21  & 2.89   & 1.95e-04  \\ & SVT & 51  & 6.60   & 1.13e-04  \\ & FPC & 91  & 15.62   & 1.78e-05  \\ \hline	
		\end{tabular} & 
		\begin{tabular}{c|c|ccc}
			\hline
			$ r $  & method & IT    & t(s)          & Rer  \\ \hline
			\multirow{4}{*}{40}  & Alg.~\ref{alg:main} & \textbf{25} & \textbf{2.77}  & \textbf{1.63e-06}  \\ & FRSI  & 28  & 3.85   & 2.90e-04  \\ & SVT & 64  & 11.15  & 1.26e-04  \\ & FPC & 125 & 56.49   & 1.83e-05  \\ \hline
			\multirow{4}{*}{80}  & Alg.~\ref{alg:main} & \textbf{31} & \textbf{7.08}  & 4.76e-05 \\ & FRSI  & 42  & 10.03  & 5.71e-04  \\ & SVT & 93  & 34.44  & 1.47e-04  \\ & FPC & 212 & 165.98  & \textbf{2.04e-05}  \\ \hline
			\multirow{4}{*}{100} & Alg.~\ref{alg:main} & \textbf{38} & \textbf{12.26} & 5.42e-05  \\ & FRSI  & 46  & 20.51  & 1.21e-04  \\ & SVT & 144 & 68.41  & 1.76e-04  \\ & FPC & 361 & 415.25  & \textbf{2.38e-05}  \\ \hline	
		\end{tabular}
	\end{tabular}
\end{table}

We also performed experiments on larger matrices with very few observed entries. The experiments were conducted under the same parameters as before and we set up a time limit of one hour. We compare the results only with SVT, because, in this case, it is faster than FRSI and FPC algorithms. 
The results are displayed in Table~\ref{tabela4}, and it can be seen that both algorithms have competitive performance for the tested cases. 
In this table we also report an additional column with the recovered rank $\hat{r}$.

Algorithm~\ref{alg:main} usually outperforms SVT in terms of relative error and it is faster for matrices with higher rank. 
SVT tends to show a better performance for smaller ranks and when the number of missing entries is not too high. 
However, it becomes considerably slow when the rank increases and the percentage of known entries decreases. 
For some cases, such as (1000,20,90\%) and (10000,40,97\%), we even had to switch to the conservative choice of $ t_{k} = 1.99 $, 
for which SVT has theoretical convergence guarantees, rather than $ t_{k} = 1.2n^{2}/ \lvert \Omega \rvert $, to avoid exceed the time limit. 
Furthermore, we remark that the rank $\hat{r}$ of the matrix recovered by SVT can be higher than the rank of the original matrix, whereas Algorithm~\ref{alg:main} recovered a matrix with correct rank for this set of experiments.


\begin{table}
	\caption{Comparison of Algorithm~\ref{alg:main} and SVT for different values of $ \left(n,r,p\right) $ and $\beta \in \{13,12,19,12,19,12,19,10\}$, respectively.}\label{tabela4}
	\centering
	\small
	\begin{tabular}{c|c|cccc}
		\hline
		$ \left(n,r,p\right) $  & method & IT    & t(s)          & Rer  & $\hat{r}$\\ \hline
		\multirow{2}{*}{(1000,10,90\%)}  & Alg.~\ref{alg:main} & \textbf{116} &  5,45  & \textbf{1.36e-04} & 10 \\ 
		& SVT & 174  & \textbf{5.26}   & 1.44e-04  & 10  \\ \hline
		\multirow{2}{*}{(1000,20,90\%)}  & Alg.~\ref{alg:main} & \textbf{102} & \textbf{7.34}   & 3.25e-01 & 20 \\ 
		& SVT & 500  & 279.56   & \textbf{2.23e-01} &  168  \\ \hline
		\multirow{2}{*}{(2000,10,90\%)}  & Alg.~\ref{alg:main} & 86 &  12.54  & \textbf{3.68e-05} & 10 \\ 
		& SVT & \textbf{83}  & \textbf{9.55}   & 1.39e-04 & 10 \\ \hline	
		\multirow{2}{*}{(2000,20,92\%)}  & Alg.~\ref{alg:main} & \textbf{147} & \textbf{28.11}   & 1.59e-04 & 20 \\ 
		& SVT & 262  & 168.62   & \textbf{1.51e-04} &  31 \\ \hline	
		\multirow{2}{*}{(5000,10,90\%)}  & Alg.~\ref{alg:main} & 69  & 63.4  & \textbf{2.36e-05} & 10 \\ 
		& SVT & \textbf{53}  & \textbf{33.7}  & 1.18e-04 & 10 \\ \hline
		\multirow{2}{*}{(5000,25,96\%)}  & Alg.~\ref{alg:main} & \textbf{215} & \textbf{149.89}  & \textbf{1.62e-04}  & 25 \\ 
		& SVT & 297  & 1355.23  & 2.34e-04 & 50 \\ \hline
		\multirow{2}{*}{(10000,10,90\%)} & Alg.~\ref{alg:main} & 65 & 245.13 &  \textbf{8.27e-06} & 10 \\ 
		& SVT & \textbf{43} & \textbf{113.84}  & 1.07e-04 & 10  \\ \hline	
		\multirow{2}{*}{(10000,40,97\%)} & Alg.~\ref{alg:main} & \textbf{256} & \textbf{1018.56} &  \textbf{8.01e-04} & 40 \\ 
		& SVT & 677 & 3600  & 4.13e-02  & 95 \\ \hline	
	\end{tabular}
\end{table}


\subsection{Experiments on \textit{MovieLens} data set}\label{sec:ML}

The \textit{MovieLens} data set is a well-known recommender system that is often used in matrix completion experiments \cite{yao2018}. It contains ratings ($ \left\lbrace 1, 2, 3, 4, 5 \right\rbrace $) of different users on movies. Table~\ref{tabela2} contains the data sets used in the experiments.

\begin{table}
	\caption{\textit{MovieLens} data sets used in the experiments}
	\centering
	\begin{tabular}{c|c|c|c}
		\hline
		data set &  \# users & \# movies & \# ratings \\ \hline
		MovieLens-100k & 943 & 1,682 & 100,000 \\
		MovieLens-1M & 6,040 & 3,952 & 1,000,209 \\
		\hline
	\end{tabular}\label{tabela2}
\end{table}

We randomly deleted $ 50\% $ percent of the observed ratings and for performance evaluation we use the root mean square error (RMSE) given by  
$$
RMSE = \sqrt{\rVert P_{\hat{\Omega}} (A-\tilde{A}) \lVert_{F}^{2}/ \lvert \hat{\Omega} \rvert},
$$
where  $\hat{\Omega}$ is total number of observed ratings (but only $\lvert\Omega\rvert = \lvert \hat{\Omega} \rvert/2$ ratings were passed as input to the algorithms). 

Since the ratings matrix has unknown rank and both Algorithm~\ref{alg:main} and FRSI need this information, we performed some experiments for different values of $ r $ and we set  $ r = 130 $ for 	MovieLens-100k and $ r=340 $ for MovieLens-1M because these choices provide the smallest RMSE for both methods. 
For Algorithm~\ref{alg:main}, we fixed the accelaration parameter $\beta=2$. 
In these experiments, we set the tolerances $ \epsilon_{\rho} =\epsilon_{1}=\epsilon_{2}=\epsilon_{3} = 10^{-3} $, $\epsilon_{\lambda}=10^{-2}$, and for SVT we fixed $ t_{k} = 1.99 $ and since $ m \neq n $ we set $ \tau = 8\sqrt{mn} $ as suggested in \cite{cai2010singular}. Moreover, we set up a time limit of one hour for all the algorithms.

The results are shown in Table \ref{tabela3}. As we can see, Algorithm~\ref{alg:main} shows the best performance in terms of CPU time and RMSE. 
We remark that for the dataset MovieLens-1M, Algorithm~\ref{alg:main} was the only one able to reach the stopping criteria in less than one hour. 

\begin{table}
	\caption{Numerical results on \textit{MovieLens} data sets}
	\centering
	\begin{tabular}{c|c|c|c|c|c|c}
		\hline
		\multirow{2}{*}{method}	         & \multicolumn{3}{c|}{100k} & \multicolumn{3}{c}{1M} \\  \cline{2-7}
		& IT        & t(s)    & RMSE     & IT         & t(s)     & RMSE  \\ \hline
		Alg.~\ref{alg:main}   & \textbf{84}        &  \textbf{61.43}  & \textbf{0.7667}   & \textbf{74}         & \textbf{967.76}   & \textbf{0.7123}  \\
		FRSI   & 223       &  159.62 & 0.8598   & 176        & 3,600     & 0.8475  \\
		SVT    & 2,000      &  1,315.92 & 0.7696  & 1236       & 3,600       & 0.7230  \\
		FPC    & 410       & 978.20 & 0.7806    & 234        & 3,600       & 0.7929 \\
		\hline
	\end{tabular}\label{tabela3}
\end{table}

\section{Conclusion}\label{sec:fim}


We consider matrix completion problems where the rank of the target matrix is known in advance. 
For instance, this is the case of localization, graph realization and other problems in distance geometry \cite{liberti2014euclidean} where the rank of the matrix to be completed is related to the embedding dimension.

We revisited a Fixed Rank Soft-Impute (FRSI) heuristic and, by analyzing the operator defining its iteration, we shown that the generated sequence is quasi-Fejér convergent to ${\cal X}^*$, under a strong assumption on the behavior of the underlying singular values. 
Nevertheless, regardless of this assumption, an accelerated version of FRSI can still be helpfull as a warm-start phase for an accelerated Soft-Impute algorithm aimed to solve a nuclear norm regularized least-squares problem. 
This idea gave rise to a two-phase rank-based algorithm (Algorithm~\ref{alg:main}) which takes into account the rank information in the heuristic of the first phase to estimate the nuclear norm regularization parameter and provide a warm starting-point to an accelerated Soft-Impute algorithm at the second phase.

After a numerical study on how to tuning parameters of the first phase, numerical experiments on both synthetic and real data sets indicates that the proposed algorithm (Alg.~\ref{alg:main}) outperforms the previous heuristic FRSI \cite{moreira2018novel} and is competitive with well established algorithms for matrix completion, such as SVT and FPC. Moreover, Algorithm~\ref{alg:main} was able to recover  low-rank matrices from a few percentage of its entries with reasonable accuracy and faster than the compared methods, mainly when the expected rank is not too low.

Even though Algorithm~\ref{alg:phase1} (phase-one) is just an heuristic, it was responsible for the majority of the iterations of Algorithm~\ref{alg:main}.
This fact points in the direction of studying convergence properties of phase-one alone under weaker assumptions yet to be discovered.

\bibliographystyle{ieeetr}
\bibliography{Bibliografia}

\appendix
\section{Soft-Impute as a proximal gradient method}\label{sec:sipg}
Here we show that Soft-Impute is a particular case of proximal gradient applied to problem (\ref{nnm}) with constant step-size. 
The optimization problem given by equation (\ref{nnm}) is a particular case of the problem of minimizing composite functions of the form
\begin{equation}\label{comp_func}
	\minimize_{x}  \quad g(x) + h(x) , 
\end{equation}
where $ g ,h $ are convex functions with $ g $ differentiable, having Lipschitz gradient $\nabla g$ with constant $L>0$ ($h$ does not need to be smooth, only proper convex).


%
%
%
%

%

Problem~\eqref{comp_func} can be solved by the proximal gradient algorithm, which generates a sequence $ \left\lbrace x^{k} \right\rbrace  $ given by
\begin{equation}\label{pga}
	x^{k+1}=\text{prox}_{th}\left(x^{k}-t\nabla g(x^{k})\right),
\end{equation}
where $ t>0 $ and $ \textrm{prox}_{th} \left(\cdot\right) $ is the proximal operator, which can be expressed as 
$$ \textrm{prox}_{th}(v) =\argmin_{x}\left\lbrace \dfrac{1}{2t}\lVert x-v \rVert_{2}^{2} +h(x) \right\rbrace. $$
It is shown (see Theorem 3.1 in \cite{fista}) that either for a fixed stepsize $ t \leq \dfrac{1}{L} $ or by a backtracking line search, the proximal algorithm converges to the optimal solution of  (\ref{comp_func}) at a rate of $O(1/k)$, where $ k $ is the number of iterations. 

For the function $ h(X) = \lambda \lVert X \rVert_{\ast} $, the proximal operator is defined as
$$	
\textrm{prox}_{th}(M)=\argmin_{X}\left\lbrace \dfrac{1}{2t}\lVert M-X \rVert_{F}^{2} + \lambda \lVert X \rVert_{\ast} \right\rbrace,
$$
whose solution is given by (see Theorem~2.1 in \cite{cai2010singular} with $ \tau=\lambda t$)
\begin{equation}\label{prox_soft}
	\textrm{prox}_{th}(M) =  S_{\lambda t}\left( M\right).
\end{equation}

In order to show that SI is a proximal gradient algorithm applied to problem~\eqref{nnm}, let us re-write iteration \eqref{eq:si} equivalently as 
\begin{align*}
	Y^{k} & = X^{k}+P_{\Omega}(A-X^{k})= P_{\Omega}(A) + P_{\Omega}^{\perp} (X^{k}) \\
	X^{k+1} & = S_{\lambda}(Y^{k}).
\end{align*}
For problem~\eqref{nnm}, observe that $ g(X) =  \dfrac{1}{2}\lVert P_{\Omega}(A)- P_{\Omega}(X)\rVert_{F}^{2} $, and thus $ L=1 $. 
Since 
$$ X^{k} - \nabla g(X^{k}) = X^{k} - (P_{\Omega}(X^{k}) - P_{\Omega}(A))  = P_{\Omega}(A) + P_{\Omega}^{\perp} (X^{k}) =: Y^{k}, $$
from (\ref{pga}) and (\ref{prox_soft}) with $ t =1 $ we have $ X^{k+1} =\textrm{prox}_{h}(Y^{k}) =  S_{\lambda }\left( Y^{k}\right)  $, which gives the result. 

One can accelerate the proximal gradient method to achieve the optimal convergence rate of $ O(1/k^{2}) $ by setting the equations \cite{fista, parikh2014proximal}
\begin{equation}\label{acpro}
	\begin{aligned}
		x^{k+1} & = \textrm{prox}_{th}\left(z^{k}-t\nabla g(z^{k})\right) \\
		z^{k+1} & = x^{k+1}+\frac{k-1}{k+2}\left(x^{k+1}-x^{k}\right) 
	\end{aligned}
\end{equation}

Therefore, in the same way Soft-Impute corresponds to a proximal gradient method with fixed step-size $t=1$, 
Algorithm~\ref{alg:phase2} corresponds to an accelerated proximal gradient, for which the convergence is well-studied in the literature \cite{parikh2014proximal}.

\end{document}